\documentclass[11pt,a4paper]{article}
\usepackage{amsmath}
\usepackage{amssymb}
\usepackage{mathrsfs}
\usepackage{hyperref}
\numberwithin{equation}{section}

\evensidemargin=\oddsidemargin\textwidth=142mm \textwidth=160mm
\newtheorem{theorem}{Theorem}[section]
\newtheorem{lemma}{Lemma}[section]

\newenvironment{proof}[1][Proof]{\begin{trivlist}
\item[\hskip \labelsep {\bfseries #1}]}{\end{trivlist}}

\begin{document}
\title {Norm of the Hilbert matrix operator between some spaces of analytic functions
\footnote{ The research was supported by  National Natural Science Foundation of China (Grant No. 11671357) and Zhejiang Provincial Natural Science Foundation of China (Grant No. LY23A010003)}}
\author{  Hao Hu\footnote{E-mail address:2716007899@qq.com} \quad\quad Shanli Ye\footnote{Corresponding author.~ E-mail address: slye@zust.edu.cn}   \\
(\small \it School of Science, Zhejiang University of Science and Technology, Hangzhou 310023, China$)$}

\date{}

\maketitle
\begin{abstract}
	In this paper, we calculate the exact value of  the norm of the Hilbert matrix operator $\mathcal{H}$ from the logarithmically weighted Korenblum space $H^\infty_{\alpha,\log}$ into Korenblum space $H^\infty_\alpha$, and from the Hardy space $H^\infty$ to the classical Bloch space $\mathcal{B}$.   Furthermore, we compute the precise value  of  the norm  on  the logarithmically weighted Korenblum space $H^\infty_{\alpha,\log}$, and obtain both the lower and upper bounds of the norm  on $\alpha$-Bloch space $\mathcal{B}^{\alpha}$.  Finally, in the context of mapping from the Korenblum space $H^\infty_\alpha$ to the $(\alpha+1)$-Bloch space $\mathcal{B}^{\alpha+1}$, we establish the norm of $\mathcal{H}$. \\
	{\small\bf Keywords}\quad {Operator norm, Hilbert matrix,  Korenblum spaces, Bloch spaces
		\\
		{\small\bf 2020 MR Subject Classification }\quad 47B38, 47B91, 30H05, 30H30, 30H99\\}

\end{abstract}
\section{Introduction}\label{s1}

\quad The Hilbert matrix operator $\mathcal{H}$ is one of the central operators in operator theory.  In recent years, the study of its boundedness, norm, and other properties on several analytic function spaces has been actively investigated. The exploration of its boundedness and norm across various analytic function spaces has been a vibrant area of research in recent years \cite{ATE,11,BNO,BTM,5,4,2,1}. Initially, Diamantopoulus and Siskakis \cite{1}  explored the boundedness of operator
 $\mathcal{H}$  on the Hardy space $H^p(1<p<\infty)$, and simultaneously derived an upper bound estimate for its norm. In a subsequent study\cite{2}, Diamantopoulus  extended the investigation to the Bergman spaces $A^p(2<p<\infty)$,  and obtained the upper bound estimate for the norm of $\mathcal{H}$. Building upon this foundational work, Dostani\'{c}, Jevti\'{c}, Vukoti\'{c} \cite{16} obtained the precise norm value of operator $\mathcal{H}$ within the Hardy space $H^p(1<p<\infty)$ and also gave exact value of the norm  of $\mathcal{H}$ in the Bergman space $A^p$ when $4\leq p<\infty$. In 2017, Bo\u{z}in and Karapetrovi\'{c}\cite{BNO} solved the problem of exact value of the norm  of $\mathcal{H}$ in the Bergman space $A^p$ for $2<p<4$.

 The study of boundedness of $\mathcal{H}$  on  the weighted Bergman spaces $A^p_\alpha$ was initiated in  \cite{7}, and some partial results were obtained. Subsequently, the norms on the weighted Bergman spaces $A^p_\alpha$ with different values of $\alpha$  have been investigated by numerous researchers. For details, we refer to \cite{11,7, 8,9,10}. In \cite{13}, the second author and Feng studied the operator $\mathcal{H}$ on the logarithmically weighted Bergman space $A_{\log^\alpha}^p$  and calculated lower bound and  upper bound for the norm of $\mathcal{H}$ from the logarithmically weighted Bergman space $A_{\log^\alpha}^2$ into Bergman space $A^2$ when $\alpha>2$.

  On the Korenblum space  $H^\infty_\alpha$,  Lindstr\"{o}m, Miihkinen and Wikman in \cite{3} calculated the precise value of the norm of $\mathcal{H}$ on  $H^\infty_\alpha$ when $0<\alpha\leq\frac{2}{3}$, and got the upper bound for the norm of $\mathcal{H}$ on  $H^\infty_\alpha$ for $\frac{2}{3}<\alpha<1$. In \cite{4}, Dai obtained  the precise value of the norm  on  $H^\infty_\alpha$ for $0<\alpha\leq\alpha_0$, where $\alpha_0> \frac{2}{3}$, and gave a better estimate of the upper bound of the norm of $\mathcal{H}$. In \cite{1}, Diamantopoulus and Siskakis discovered that the Hilbert matrix operator $\mathcal{H}$  is not bounded on the Hardy space  $H^\infty$. From \cite{6,3}, we know that  it is bounded operator from   $H^\infty$ into the classical Bloch space $\mathcal{B}$, and bounded on  the Korenblum space $H^\infty_\alpha$ for $0<\alpha<1$.

  In this article, we present the norm of $\mathcal{H}$  between certain spaces of analytic functions. Our paper is organized as follows.
  In Section 2, we introduce the notion of the logarithmically weighted Korenblum space and some notations.
  In Section 3,  when $0<\alpha<1$, we calculate the exact value of the norm from the logarithmically weighted Korenblum space  $H^\infty_{\alpha,\log}$ to the Korenblum space $H^\infty_\alpha$. This value is represented by the supremum of an integral, which makes it hard to determine its magnitude. Thus, we also offer its upper and lower bounds.
 In Section 4, we calculate the exact value of  the norm  on  the logarithmically weighted Korenblum space $H^\infty_{\alpha,\log}$. In Section 5,  when $1<\alpha< 2$ , we obtain both the lower and upper bounds of  the norm  on $\alpha$-Bloch space $\mathcal{B}^{\alpha}$. Additionally we show that it is not bounded on $\mathcal{B}^{\alpha}$ for $\alpha\geq 2$, or $0<\alpha\leq 1$. In Section 6, it is found that the norm of $\mathcal{H}$ from the Hardy space $H^\infty$ to the classical Bloch space $\mathcal{B}$ is equal to 3. In Section 7, for $0<\alpha\leq \frac 23$, we obtain  the exact value of  the norm of  $\mathcal{H}$ from the  Korenblum space $H^\infty_\alpha$ to $(\alpha+1)$-Bloch space $\mathcal{B}^{\alpha+1}$. For $\frac 23<\alpha< 1$, we also  give both the lower and upper bounds. Moreover, we prove that  $\mathcal{H}: H^\infty_\alpha\to \mathcal{B}^{\alpha+1}$ is not bounded when $\alpha\geq 1$.

\section{Preliminaries}

Let $\mathbb{D}$ denote the open unit disk of the complex plane $\mathbb{C}$, and let $H(\mathbb{D})$ denote the set of all analytic functions in  $\mathbb{D}$.

For $0 < p \leq \infty $, the Hardy space $H^p$ is the space comprising all functions  $f \in H(\mathbb{D} )$ such that
$$\|f\|_{H^p}=\sup_{0\leq r <1} M_p(r,f)<\infty,$$
where
$$M_p(r,f)=\left( \frac{1}{2\pi}\int_0^{2\pi} |f(re^{it})|^p dt\right)^{\frac{1}{p}}, \quad 0<p<\infty;$$
$$M_\infty(r,f)=\sup_{0\leq t<2\pi}|f(re^{it})|.$$

We refer to \cite{Dur} for the terminology and findings on Hardy spaces.

For $0 <\alpha<1 $, the Korenblum space \cite{3} $H^\infty_\alpha$ is the space of all functions $f\in H(\mathbb{D} )$ with
$$\|f\|_{H^\infty_\alpha}=\sup_{z\in\mathbb{D}}(1-|z|^2)^\alpha|f(z)|<\infty.$$

Now  we define   the logarithmically weighted Korenblum spaces $H^\infty_{\alpha,\log}$, which consists of those $f \in H(\mathbb{D})$ such that
$$\|f\|_{H^\infty_{\alpha,\log}}\overset{def}{=} \sup_{z\in\mathbb{D}}(1-|z|^2)^\alpha\log\frac{2e^\frac{1}{\alpha}}{1-|z|^2}|f(z)|<\infty,  $$
It is easily obtained that $H^\infty \subsetneqq H^\infty_{\alpha,\log} \subsetneqq H^\infty_\alpha$.

For $0<\alpha<\infty$, the $\alpha$-Bloch space $\mathcal{B}^{\alpha}$  consists of those functions $f\in H(\mathbb{D})$ with
$$\|f\|_{\mathcal{B}^\alpha}= |f(0)|+\sup_{z\in\mathbb{D}}(1-|z|^2)^{\alpha}|f'(z)|<\infty.$$

We can see that $\mathcal{B}^1$ is the classical Bloch space $\mathcal{B}$.  We mention \cite{pommerenke,Zhu} as general references for the classical Bloch space and  the $\alpha$-Bloch spaces.

The Hilbert matrix is an infinite matrix $\mathcal{H}$ whose entries are $a_{n,k} = \frac{1}{n+k+1} $, $n,k \geq 0$. The Hilbert matrix $\mathcal{H}$ can be also viewed as an operator on spaces of analytic functions by its action on their Taylor coefficients. Hence for those $f \in H(\mathbb{D})$, $f(z) = \sum_{k=0}^\infty a_kz^k$, then we define a transformation $\mathcal{H}$ by
\begin{align}
	\mathcal{H}f(z)=\sum_{n=0}^\infty \left( \sum_{k=0}^\infty \frac{a_k}{n+k+1} \right)z^n,\notag
\end{align}
whenever the right hand side makes sense and  defines an analytic function in $\mathbb{D}$.
\section{Norm  of the Hilbert matrix $\|\mathcal{H}\|_{H^\infty_{\alpha,\log} \rightarrow H^\infty_\alpha}$}

For $0<\alpha<1$, since $H^\infty_{\alpha,\log} \subsetneqq H^\infty_\alpha$ and the norm of the Hilbert matrix operator $\mathcal{H}$  on the Korenblum space $H^\infty_\alpha$ is bounded within this range, we can conclude that $\mathcal{H}$ is a bounded operator  from $H^\infty_{\alpha,\log}$ into $H^\infty_\alpha$. In this section,
our aim is to derive norm estimates for $\mathcal{H}$ as it acts from $H^\infty_{\alpha,\log}$ into $H^\infty_\alpha$ for $0 <\alpha<1$. We know from \cite{1} that $\mathcal{H}f(z)$ is a well-defined analytic function on the unit disk $\mathbb{D}$ for $ f\in H^\infty_\alpha$. Consequently, we can assert that
\begin{align} \label{eq2.1}
	\mathcal{H}f(z) &= \sum_{n=0}^\infty \left( \sum_{k=0}^\infty \frac{a_k}{n+k+1} \right)z^n  \notag\\
	&= \sum_{n=0}^\infty \left( \sum_{k=0}^\infty a_k \int_0^1 t^{n+k}dt \right)z^n
	\notag\\
	&= \int_0^1 \sum_{k=0}^\infty a_k  t^k \sum_{n=0}^\infty  t^n z^n dt \notag \\
	&= \int_0^1 \frac{f(t)}{1-tz} dt.
\end{align}
We know that the Hilbert matrix operator $\mathcal{H}$ has an integral representation in terms of
weighted composition operators $T_t$ (see \cite{1}):
\begin{align}\label{eq2.2}
	\mathcal{H} f(z)=\int_0^1 T_t f(z) dt,
\end{align}
where
\begin{align}
	T_tf(z)=w_t(z)f(\phi_t(z)),  \quad w_t(z)=\frac{1}{1-(1-t)z}, \quad \phi_t(z)=\frac{t}{1-(1-t)z}. \notag
\end{align}

\begin{theorem}\label{Th2.1}
	For $0<\alpha<1$, then
$$ \|\mathcal{H}\|_{H^\infty_{\alpha,\log} \rightarrow H^\infty_\alpha}=\sup_{0\leq r<1}\int_{0}^{1}\frac{(1+r)^\alpha\left((t-1)r+1\right)^{2\alpha-1}}{(1-t)^\alpha \left((t-1)r+1+t\right)^\alpha\log\frac{2e^\frac{1}{\alpha}}{1-\left(\frac{t}{1-(1-t)r}\right)^2}}dt,$$
and
$$\|\mathcal{H}\|_{H^\infty_{\alpha,\log} \rightarrow H^\infty_\alpha}\geq\int_{0}^{1}\frac{1}{(1-t^2)^\alpha\log\frac{2e^\frac{1}{\alpha}}{1-t^2}}dt
	.$$

\end{theorem}
\begin{proof}
Let $0<\alpha<1$ and $z \in \mathbb{D}$. Define$$f_\alpha(z)=\frac{1}{(1-z^2)^\alpha\log\frac{2e^\frac{1}{\alpha}}{1-z^2}}.$$
By a simple calculation, we know that $g(x)=x^\alpha\log\frac{2e^\frac{1}{\alpha}}{x}$ is monotonically increasing in $(0,2)$. Since $0\leq |1-z^2|\leq2$, we obtained that
\begin{align*}
\|f_\alpha\|_{H^\infty_{\alpha,\log} }
&=\sup_{z\in\mathbb{D}}(1-|z|^2)^\alpha\log\frac{2e^\frac{1}{\alpha}}{1-|z|^2}\left|\frac{1}{(1-z^2)^\alpha\log\frac{2e^\frac{1}{\alpha}}{1-z^2}}\right|\\
&\leq\sup_{z\in\mathbb{D}}(1-|z|^2)^\alpha\log\frac{2e^\frac{1}{\alpha}}{1-|z|^2}\frac{1}{(1-|z|^2)^\alpha\log\frac{2e^\frac{1}{\alpha}}{1-|z|^2}}\\
&=\sup_{0\leq r<1}(1-r^2)^\alpha\log\frac{2e^\frac{1}{\alpha}}{1-r^2}\frac{1}{(1-r^2)^\alpha\log\frac{2e^\frac{1}{\alpha}}{1-r^2}}\\
&=1.
\end{align*}

Since $\lim\limits_{|z|\rightarrow1}|f_\alpha(z)|(1-|z|^2)^\alpha\log\frac{2e^\frac{1}{\alpha}}{1-|z|^2}=1$,  we obtain $\|f_\alpha\|_{H^\infty_{\alpha,\log} }=1.$

The weighted composition operator $T_t$ applied to a function $f_\alpha$ can be written as
\begin{align*}
	T_tf_\alpha(z)&=\frac{1}{1-(1-t)z}f_\alpha\left(\frac{t}{1-(1-t)z}\right) \\
	&=\frac{1}{1-(1-t)z}\frac{1}{\left(1-\left(\frac{t}{1-(1-t)z}\right)^2\right)^\alpha\log\frac{2e^\frac{1}{\alpha}}{1-\left(\frac{t}{1-(1-t)z}\right)^2}}.
\end{align*}
Since $\mathcal{H} f_\alpha(z)=\int_0^1 T_t f_\alpha(z) dt$ and $\|f_\alpha\|_{H^\infty_{\alpha,\log} }=1,$ we get that
\begin{align}\label{eq2.3}
    \|\mathcal{H}\|_{H^\infty_{\alpha,\log} \rightarrow H^\infty_\alpha}&\geq\|\mathcal{H}f_\alpha\|_{H^\infty_\alpha}=\sup_{z\in\mathbb{D}}(1-|z|^2)^\alpha|\mathcal{H}f_\alpha(z)|\notag\\
    &=\sup_{z\in\mathbb{D}}\int_{0}^{1}(1-|z|^2)^\alpha\frac{1}{1-(1-t)z}\frac{1}{\left(1-\left(\frac{t}{1-(1-t)z}\right)^2\right)^\alpha\log\frac{2e^\frac{1}{\alpha}}{1-\left(\frac{t}{1-(1-t)z}\right)^2}}dt\notag\\
    &\geq\sup_{0\leq r<1}\int_{0}^{1}(1-r^2)^\alpha\frac{1}{1-(1-t)r}\frac{1}{\left(1-\left(\frac{t}{1-(1-t)r}\right)^2\right)^\alpha\log\frac{2e^\frac{1}{\alpha}}{1-\left(\frac{t}{1-(1-t)r}\right)^2}}dt\notag\\
    &=\sup_{0\leq r<1}\int_{0}^{1}\frac{(1+r)^\alpha\left((t-1)r+1\right)^{2\alpha-1}}{(1-t)^\alpha \left((t-1)r+1+t\right)^\alpha\log\frac{2e^\frac{1}{\alpha}}{1-\left(\frac{t}{1-(1-t)r}\right)^2}}dt.
\end{align}
On the other hand, we have that

\begin{align}\label{eq2.4}
	|T_tf(z)|&=|w_t(z)f(\phi_t(z))|\notag\\
	&= \lvert w_t(z)\rvert\frac{1}{(1-|\phi_t(z)|^2)^\alpha\log\frac{2e^\frac{1}{\alpha}}{1-|\phi_t(z)|^2}}(1-|\phi_t(z)|^2)^\alpha\log\frac{2e^\frac{1}{\alpha}}{1-|\phi_t(z)|^2}|f(\phi_t(z))|\notag\\
	&\leq \lvert w_t(z)\rvert\frac{1}{(1-|\phi_t(z)|^2)^\alpha\log\frac{2e^\frac{1}{\alpha}}{1-|\phi_t(z)|^2}}\|f\|_{H^\infty_{\alpha,\log} }.
\end{align}
    Noting that $|w_t(z)|=\frac{1}{|1-(1-t)z|}\leq\frac{1}{1-(1-t)|z|}$, $|\phi_t(z)|=\frac{t}{|1-(1-t)z|}\leq\frac{t}{1-(1-t)|z|}$ and the monotone property of $g$, we get $$|T_tf(z)|\leq w_t(|z|)\frac{1}{(1-\phi_t(|z|)^2)^\alpha\log\frac{2e^\frac{1}{\alpha}}{1-\phi_t(|z|)^2}}\|f\|_{H^\infty_{\alpha,\log}}.$$
    Therefore, we can obtain that
\begin{align}\label{eq2.5}
     \|\mathcal{H}f\|_{H^\infty_\alpha}&=\sup_{z\in\mathbb{D}}(1-|z|^2)^\alpha\int_{0}^{1}T_tf(z)dt\notag\\
     &\leq\sup_{0\leq r<1}(1-r^2)^\alpha\int_{0}^{1}\frac{1}{1-(1-t)r}\frac{1}{\left(1-\left(\frac{t}{1-(1-t)r}\right)^2\right)^\alpha\log\frac{2e^\frac{1}{\alpha}}{1-\left(\frac{t}{1-(1-t)r}\right)^2}}dt\|f\|_{H^\infty_{\alpha,\log}}\notag\\
     &=\sup_{0\leq r<1}\int_{0}^{1}\frac{(1+r)^\alpha\left((t-1)r+1\right)^{2\alpha-1}}{(1-t)^\alpha \left((t-1)r+1+t\right)^\alpha\log\frac{2e^\frac{1}{\alpha}}{1-\left(\frac{t}{1-(1-t)r}\right)^2}}dt\|f\|_{H^\infty_{\alpha,\log}}.
\end{align}
 Then we obtain that
 \begin{align}\label{eq2.6}
 	\|\mathcal{H}\|_{H^\infty_{\alpha,\log} \rightarrow H^\infty_\alpha}\leq\sup_{0\leq r<1}\int_{0}^{1}\frac{(1+r)^\alpha\left((t-1)r+1\right)^{2\alpha-1}}{(1-t)^\alpha \left((t-1)r+1+t\right)^\alpha\log\frac{2e^\frac{1}{\alpha}}{1-\left(\frac{t}{1-(1-t)r}\right)^2}}dt<\infty.
 \end{align}
 Noting that
 \begin{align*}
        \lim_{r\rightarrow0}\int_{0}^{1}\frac{(1+r)^\alpha\left((t-1)r+1\right)^{2\alpha-1}}{(1-t)^\alpha \left((t-1)r+1\right)^\alpha\log\frac{2e^\frac{1}{\alpha}}{1-\left(\frac{t}{1-(1-t)r+t}\right)^2}}dt=\int_{0}^{1}\frac{1}{(1-t^2)^\alpha\log\frac{2e^\frac{1}{\alpha}}{1-t^2}}dt,
	\end{align*}
 We  finish the proof of the theorem.
\end{proof}

\begin{lemma} \cite{3} \label{Le2.1}
	Let $\frac{1}{2}<\alpha<1$, then\\
	$\sup_{z\in\mathbb{D}} |1-(1-t)z|^{2\alpha-1}\left(\frac{(1-|z|^2)}{|1-(1-t)z|^2-t^2}\right)^\alpha=$
    \begin{align*}
	\begin{cases}
		&\frac{t^{\alpha-1}}{(1-t)^\alpha}, if\frac{1}{2}<\alpha\leq\frac{2}{3}\ and\ 0<t<1,\  or\ if \frac{2}{3}<\alpha<1\ and \ \frac{3\alpha-2}{4\alpha-2}\leq t<1,\\
	    &(1-x_0)^{2\alpha-1}\left(\frac{1-|\frac{x_0}{1-t}|^2}{(1-x_0)^2-t^2}\right)^\alpha, if\ \frac{2}{3}<\alpha<1  and \ 0< t<\frac{3\alpha-2}{4\alpha-2},
	\end{cases}
    \end{align*}
	where$$x_0=\frac{\alpha+2\alpha t-t-\sqrt{4\alpha^2t-2\alpha t+\alpha^2-2\alpha+1}}{2\alpha-1}.$$
\end{lemma}

\begin{lemma}\label{Le2.2}
Let $0<\alpha<1$, then

\begin{align*}
	\|T_t\|_{H^\infty_{\alpha,\log} \rightarrow H^\infty_\alpha}\leq
	\begin{cases}
		&\frac{t^{\alpha-1}}{(1-t)^\alpha\log\frac{(2-t)^2e^{\frac{1}{\alpha}}}{2-2t}}, if0<\alpha\leq\frac{2}{3}\ and\ 0<t<1, \ or\ if \frac{2}{3}<\alpha<1\ and \ \frac{3\alpha-2}{4\alpha-2}\leq t<1,\\
		&\frac{(1-x_0)^{2\alpha-1}}{\log\frac{(2-t)^2e^{\frac{1}{\alpha}}}{2-2t}}\left(\frac{1-|\frac{x_0}{1-t}|^2}{(1-x_0)^2-t^2}\right)^\alpha, if\ \frac{2}{3}<\alpha<1   and \ 0< t<\frac{3\alpha-2}{4\alpha-2},
	\end{cases}
\end{align*}
where$$x_0=\frac{\alpha+2\alpha t-t-\sqrt{4\alpha^2t-2\alpha t+\alpha^2-2\alpha+1}}{2\alpha-1}.$$
\end{lemma}
\begin{proof}
	Assume first that $0<\alpha\leq\frac{1}{2}$ and $f\in H^\infty_{\alpha,\log}$. Then, the estimate $|\phi'_t(z)|\leq\frac{1-t}{t}$, $\frac{1}{\log\frac{2e^\frac{1}{\alpha}}{1-|\phi_t(z)|^2}}\leq\frac{1}{\log\frac{(2-t)^2e^{\frac{1}{\alpha}}}{2-2t}}$ and the Schwarz-Pick lemma yield that
	\begin{align}\label{eq2.7}
			\|T_t(f)\|_{ H^\infty_\alpha}&=\sup_{z\in\mathbb{D}}|T_tf(z)|(1-|z|^2)^\alpha\notag\\
			&=\sup_{z\in\mathbb{D}}\frac{1}{\sqrt{t(1-t)}}|\phi'_t(z)|^{\frac{1}{2}-\alpha}|f(\phi_t(z))|(1-|z|^2)^\alpha|\phi'_t(z)|^\alpha\notag\\
			&=\sup_{z\in\mathbb{D}}\frac{1}{\sqrt{t(1-t)}}\frac{|\phi'_t(z)|^{\frac{1}{2}-\alpha}}{\log\frac{2e^\frac{1}{\alpha}}{1-|\phi_t(z)|^2}}|f(\phi_t(z))|(1-|z|^2)^\alpha|\phi'_t(z)|^\alpha\log\frac{2e^\frac{1}{\alpha}}{1-|\phi_t(z)|^2}\notag\\
			&\leq\frac{1}{\sqrt{t(1-t)}}\left(\frac{1-t}{t}\right)^{\frac{1}{2}-\alpha}\frac{1}{\log\frac{(2-t)^2e^{\frac{1}{\alpha}}}{2-2t}}\sup_{z\in\mathbb{D}}|f(\phi_t(z))|(1-|\phi_t(z)|^2)^\alpha\log\frac{2e^\frac{1}{\alpha}}{1-|\phi_t(z)|^2}\notag\\
			&\leq\frac{t^{\alpha-1}}{(1-t)^\alpha\log\frac{(2-t)^2e^{\frac{1}{\alpha}}}{2-2t}}\|f\|_{H^\infty_{\alpha,\log}}.
	\end{align}
	We get the result for $0<\alpha\leq\frac{1}{2}$.\\
	Let $\frac{1}{2}<\alpha<1$. According to \cite{5}, we can get that
	\begin{align}\label{eq2.8}
		\|T_t\|_{H^\infty_{\alpha,\log} \rightarrow H^\infty_\alpha}&=\sup_{z\in\mathbb{D}}\frac{1}{|1-(1-t)z|}\frac{(1-|z|^2)^\alpha}{(1-\frac{t^2}{|1-(1-t)z|^2})^\alpha\log\frac{2e^\frac{1}{\alpha}}{1-\frac{t^2}{|1-(1-t)z|^2}}}\notag\\
		&\leq\frac{1}{\log\frac{(2-t)^2e^{\frac{1}{\alpha}}}{2-2t}}\sup_{z\in\mathbb{D}}\frac{1}{|1-(1-t)z|}\frac{(1-|z|^2)^\alpha}{(1-\frac{t^2}{|1-(1-t)z|^2})^\alpha}.
	\end{align}
	According to lemma \ref{Le2.1}, we can get the result.
\end{proof}
\begin{theorem}\label{Th2.3}
	For $0<\alpha\leq\frac{2}{3}$,  then
    $$\|\mathcal{H}\|_{H^\infty_{\alpha,\log} \rightarrow H^\infty_\alpha}\leq
	\int_{0}^{1}\frac{t^{\alpha-1}}{(1-t)^\alpha\log\frac{(2-t)^2e^{\frac{1}{\alpha}}}{2-2t}}dt.$$
	
For $\frac{2}{3}<\alpha<1$, then
	$$\|\mathcal{H}\|_{H^\infty_{\alpha,\log} \rightarrow H^\infty_\alpha}\leq\int_{0}^{\frac{3\alpha-2}{4\alpha-2}} \frac{G(x_0)}{\log\frac{(2-t)^2e^{\frac{1}{\alpha}}}{2-2t}}dt+\int_{\frac{3\alpha-2}{4\alpha-2}}^{1}\frac{t^{\alpha-1}}{(1-t)^\alpha\log\frac{(2-t)^2e^{\frac{1}{\alpha}}}{2-2t}}dt, $$
	where$$G(x)=(1-x)^{2\alpha-1}\left(\frac{1-|\frac{x}{1-t}|^2}{(1-x)^2-t^2}\right)^\alpha$$
	and$$x_0=\frac{\alpha+2\alpha t-t-\sqrt{4\alpha^2t-2\alpha t+\alpha^2-2\alpha+1}}{2\alpha-1}.$$
\end{theorem}
\begin{proof}
	Let $0<\alpha<1$, we have that
	\begin{align}\label{eq2.9}
		\|\mathcal{H}f\|_{ H^\infty_\alpha}&=\sup_{z\in\mathbb{D}}|\int_{0}^{1}T_tf(z)dt(1-|z|^2)^\alpha|\notag\\
		&\leq\int_{0}^{1}\sup_{z\in\mathbb{D}}|T_tf(z)dt(1-|z|^2)^\alpha|dt\notag\\
		&\leq\int_{0}^{1}\|T_t\|_{H^\infty_{\alpha,\log} \rightarrow H^\infty_\alpha}dt.
	\end{align}
	We will combine lemma \ref{Le2.2} and (\ref{eq2.8}) to obtain the result.
\end{proof}
\section{Norm  of the Hilbert matrix $\|\mathcal{H}\|_{H^\infty_{\alpha,\log} \rightarrow H^\infty_{\alpha,\log}}$}
In this section,  we calculate the norm of the Hilbert matrix operator  $\mathcal{H}$ acting on   the logarithmically weighted Korenblum space $H^\infty_{\alpha,\log}$.

 We Know that the Beta function defined as
$$B(s,t)=\int_{0}^{1}x^{s-1}(1-x)^{t-1}dx,$$
where $s,t\in\mathbb{C}/\mathbb{Z}$ satisfing $R(t) > 0$ and $R(s) > 0$. It can be checked that$$B(s, t) =\Gamma(s)\Gamma(t)/\Gamma(s+t),$$
where $\Gamma$ is the Gamma function. We will also use the well-known equation
$$\Gamma(z)\Gamma(1-z)=\frac{\pi}{\sin(\pi z)},z\in\mathbb{C}/\mathbb{Z}.$$
We can easily obtain that $$\int_{0}^{1}\frac{t^{\alpha-1}}{(1-t)^\alpha}dt=\frac{\pi}{\sin\alpha\pi}.$$

\begin{theorem}
	For $0<\alpha<1$, then the Hilbert matrix operator $\mathcal{H}$ is bounded on the logarithmically weighted Korenblum space  $H^\infty_{\alpha,\log}$. Moreover, the norm of $\mathcal{H}$ satisfies the following equations:
$$\|\mathcal{H}\|_{H^\infty_{\alpha,\log} \rightarrow H^\infty_{\alpha,\log}}=\sup_{0\leq r<1}\int_{0}^{1}\frac{(1+r)^{\alpha}\left((t-1)r+1\right)^{2\alpha-1}\log\frac{2e^\frac{1}{\alpha}}{1-r^2}}{(1-t)^{\alpha}\left((t-1)r+1+t\right)^{\alpha}
\log\frac{2e^\frac{1}{\alpha}}{1-\left(\frac{t}{1-(1-t)r}\right)^2}}dt,$$
and
 $$\|\mathcal{H}\|_{H^\infty_{\alpha,\log} \rightarrow {H^\infty_{\alpha,\log} }}
	\geq\frac{\pi}{\sin\alpha\pi}.$$
\end{theorem}
\begin{proof}
	Let   $$f_\alpha(z)=\frac{1}{(1-z^2)^\alpha\log\frac{2e^\frac{1}{\alpha}}{1-z^2}}.$$
    From the proof of Theorem 3.1, we know that $\|f_\alpha\|_{H^\infty_{\alpha,\log} }=1$ and
    $$T_tf_\alpha(z)=\frac{1}{1-(1-t)z}\frac{1}{\left(1-\left(\frac{t}{1-(1-t)z}\right)^2\right)^\alpha\log\frac{2e^\frac{1}{\alpha}}{1-\left(\frac{t}{1-(1-t)z}\right)^2}}.$$
Then
    \begin{align}\label{eq3.1}
    	 \|\mathcal{H}\|_{H^\infty_{\alpha,\log} \rightarrow {H^\infty_{\alpha,\log} }}&\geq\|\mathcal{H}f_\alpha\|_{H^\infty_{\alpha,\log}}=\sup_{z\in\mathbb{D}}(1-|z|^2)^\alpha\log\frac{2e^\frac{1}{\alpha}}{1-|z|^2}|\mathcal{H}f_\alpha(z)|\notag\\
    	 &\geq\sup_{0\leq r<1}(1-r^2)^\alpha\log\frac{2e^\frac{1}{\alpha}}{1-r^2}|\mathcal{H}f_\alpha(r)|\notag\\
    	 &=\sup_{0\leq r<1}\int_{0}^{1}\frac{(1+r)^{\alpha}\left((t-1)r+1\right)^{2\alpha-1}\log\frac{2e^\frac{1}{\alpha}}{1-r^2}}{(1-t)^{\alpha}\left((t-1)r+1+t\right)^{\alpha}\log\frac{2e^\frac{1}{\alpha}}{1-\left(\frac{t}{1-(1-t)r}\right)^2}}dt.
    \end{align}

    Next, we will consider the upper bound for the norm. According to (\ref{eq2.4}), we have that	
    \begin{align}\label{eq3.2}
    \|\mathcal{H}f\|_{H^\infty_{\alpha,\log}}&=\sup_{z\in\mathbb{D}}(1-|z|^2)^\alpha\log\frac{2e^\frac{1}{\alpha}}{1-|z|^2}|\int_{0}^{1}T_tf(z)dt|\notag\\
    &\leq\sup_{z\in\mathbb{D}}(1-|z|^2)^\alpha\log\frac{2e^\frac{1}{\alpha}}{1-|z|^2}\int_{0}^{1}|T_tf(z)|dt\notag\\
    &\leq\sup_{z\in\mathbb{D}}(1-|z|^2)^\alpha\log\frac{2e^\frac{1}{\alpha}}{1-|z|^2}\int_{0}^{1}\lvert w_t(z)\rvert\frac{1}{(1-|\phi_t(z)|^2)^\alpha\log\frac{2e^\frac{1}{\alpha}}{1-|\phi_t(z)|^2}}\|f\|_{H^\infty_{\alpha,\log} }dt\notag\\
    &\leq\sup_{0\leq r<1}(1-r^2)^\alpha\log\frac{2e^\frac{1}{\alpha}}{1-r^2}\int_{0}^{1}\lvert w_t(r)\rvert\frac{1}{(1-\phi^2_t(r))^\alpha\log\frac{2e^\frac{1}{\alpha}}{1-\phi^2_t(r)}}\|f\|_{H^\infty_{\alpha,\log} }dt\notag\\
    &=\sup_{0\leq r<1}\int_{0}^{1}\frac{(1+r)^{\alpha}\left((t-1)r+1\right)^{2\alpha-1}\log\frac{2e^\frac{1}{\alpha}}{1-r^2}}{(1-t)^{\alpha}\left((t-1)r+1+t\right)^{\alpha}\log\frac{2e^\frac{1}{\alpha}}{1-\left(\frac{t}{1-(1-t)r}\right)^2}}dt.
    \end{align}
    According to (\ref{eq3.1}) and (\ref{eq3.2}), we obtain that
    $$\|\mathcal{H}\|_{H^\infty_{\alpha,\log} \rightarrow H^\infty_{\alpha,\log}}=\sup_{0\leq r<1}\int_{0}^{1}\frac{(1+r)^{\alpha}\left((t-1)r+1\right)^{2\alpha-1}\log\frac{2e^\frac{1}{\alpha}}{1-r^2}}{(1-t)^{\alpha}\left((t-1)r+1+t\right)^{\alpha}
\log\frac{2e^\frac{1}{\alpha}}{1-\left(\frac{t}{1-(1-t)r}\right)^2}}dt<\infty,$$
  then   $\mathcal{H}$ is bounded on $H^\infty_{\alpha,\log}$.

       Noting that$\lim\limits_{r\rightarrow1}\frac{\log\frac{2e^\frac{1}{\alpha}}{1-r^2}}{\log\frac{2e^\frac{1}{\alpha}}{1-\left(\frac{t}{1-(1-t)r}\right)^2}}=1$, we easily obtain that
    \begin{align*}
    \|\mathcal{H}\|_{H^\infty_{\alpha,\log} \rightarrow {H^\infty_{\alpha,\log} }}
    &\geq\lim\limits_{r\rightarrow1}\int_{0}^{1}\frac{(1+r)^{\alpha}\left((t-1)r+1\right)^{2\alpha-1}\log\frac{2e^\frac{1}{\alpha}}{1-r^2}}{(1-t)^{\alpha}\left((t-1)r+1+t\right)^{\alpha}\log\frac{2e^\frac{1}{\alpha}}{1-\left(\frac{t}{1-(1-t)r}\right)^2}}dt\\
    &=\int_{0}^{1}\frac{t^{\alpha-1}}{(1-t)^\alpha}dt=\frac{\pi}{\sin\alpha\pi}.
    \end{align*}
    Thus we calculated the lower bound for the norm.
\end{proof}
\section{Norm estimates of the Hilbert matrix $\|\mathcal{H}\|_{\mathcal{B}^\alpha\rightarrow \mathcal{B}^\alpha}$}
 In this section, we explore the norm estimates for the Hilbert matrix operator $\mathcal{H}$ acting on $\mathcal{B}^\alpha$ for $1<\alpha<2$. For convenience, we set
  $$\|f\|_{\alpha*}=\sup_{z\in\mathbb{D}}(1-|z|^2)^{\alpha}|f'(z)|.$$

  According to (\ref{eq2.1}), we obtain that
\begin{align}\label{eq5.1}
	(\mathcal{H}f)'(z)=\int_{0}^{1}\frac{t}{(1-tz)^2}f(t)dt.
\end{align}
For $z\in\mathbb{D}$, we can choose the path in \cite{1}$$\zeta(t)=\zeta_z(t)=\frac{t}{(t-1)z+1}.$$
The change of variable in (\ref{eq5.1}) gives
\begin{align}\label{eq5.2}
	(\mathcal{H}f)'(z)=\int_{0}^{1}\frac{t}{\left[(t-1)z+1\right](1-z)}f(\phi_t(z))dt.
\end{align}
 \begin{lemma}
 Let $f\in\mathcal{B}^\alpha$, then
  \begin{align*}
	|f(z)|\leq
	\begin{cases}
		&\frac{(1-|z|)^{1-\alpha}-1}{\alpha-1}\|f\|_{\alpha*}+|f(0)|, if  \alpha\not=1,\\
		&\log\frac{1}{1-|z|}\|f\|_{\alpha*}+|f(0)|,  if  \alpha =1.\\
	\end{cases}
\end{align*}

 \end{lemma}

 \begin{proof}
 	Suppose $f\in\mathcal{B}^\alpha$ and $z\in\mathbb{D}$, then
 	\begin{align*}
 		|f(z)-f(zt)|&=\left|z\int_{t}^{1}f'(tz)dt\right|\\
 		&\leq\|f\|_{\alpha*}\int_{t}^{1}\frac{|z|}{(1-|tz|^2)^\alpha}dt\\
 		&\leq\|f\|_{\alpha*}\int_{|z|t}^{|z|}\frac{1}{(1-x)^\alpha}dx.
 	\end{align*}
 	Especially, let $t=0$, we have $|f(z)-f(0)|\leq\|f\|_{\alpha*}\int_{0}^{|z|}\frac{1}{(1-x)^\alpha}dx.$
 	By a simple calculation, we have obtained the result.
 \end{proof}
  \begin{theorem}
 	For $1<\alpha<2$, we obtain that
 	$$\|\mathcal{H}\|_{\mathcal{B}^\alpha\rightarrow \mathcal{B}^\alpha}\geq\int_{0}^{1}\frac{(1-t^2)^{1-\alpha}}{2(\alpha-1)}dt-\frac{1}{2(\alpha-1)}+\frac{\pi}{\sin\pi(\alpha-1)},$$
 	Moreover,  $\mathcal{H}$ is not bounded on $\mathcal{B}^\alpha$ when $0<\alpha\leq1$, or $\alpha\geq2$.
 \end{theorem}
 \begin{proof}
 	Let $\alpha\not=1$ and $z\in\mathbb{D}$. Define$$h_\alpha(z)=\frac{1}{2(\alpha-1)}(1-z^2)^{1-\alpha}-\frac{1}{2(\alpha-1)}$$
 	On one hand, we have the estimate
 	$$\|h_\alpha\|_{\mathcal{B}^\alpha}=\sup_{z\in\mathbb{D}}\frac{|z|(1-|z|^2)^\alpha}{|1-z^2|^\alpha}\leq\sup_{z\in\mathbb{D}}|z|=1$$
 	On the other hand, for $r\in(0,1)$ it holds that $\lim_{r\rightarrow1^-}|h'_\alpha(r)|(1-r^2)^\alpha=1$,  and we obtain that $\|h_\alpha\|_{\mathcal{B}^\alpha}=1.$
 	
 	According to (\ref{eq5.2}), we obtain that
 	
 	\begin{align*}
 		\|\mathcal{H}\|_{\mathcal{B}^\alpha\rightarrow\mathcal{B}^\alpha}&\geq\frac{\|\mathcal{H}h_\alpha\|_{\mathcal{B}^\alpha}}{\|h_\alpha\|_{\mathcal{B}^\alpha}}\\
 		&=|\mathcal{H}h_{\alpha}(0)|+\sup_{z\in\mathbb{D}}|(\mathcal{H}h_{\alpha})'(z)|(1-|z|^2)^\alpha\\
 		&=\frac{1}{2(\alpha-1)}\left(\int_{0}^{1}(1-t^2)^{1-\alpha}dt-1\right)+\frac{1}{2|\alpha-1|}\sup_{z\in\mathbb{D}}(1-|z|^2)^\alpha\left|\int_{0}^{1}\frac{t(1-\phi_t(z)^2)^{1-\alpha}-t}{\left[(t-1)z+1\right](1-z)}dt\right|\\
 		&=\int_{0}^{1}\frac{1}{2(\alpha-1)}(1-t^2)^{1-\alpha}dt-\frac{1}{2(\alpha-1)}\\
 		&+\frac{1}{2|\alpha-1|}\sup_{z\in\mathbb{D}}(1-|z|^2)^\alpha\left|\int_{0}^{1}\frac{t(1-t)^{1-\alpha}\left[(t-1)z+1+t\right]^{1-\alpha}}{(1-z)^\alpha\left[(t-1)z+1\right]^{3-2\alpha}}-\frac{t}{\left[(t-1)z+1\right](1-z)}dt\right|.
 	\end{align*}
    For $0<\alpha<1$, we have that
    \begin{align*}
    	&\sup_{z\in\mathbb{D}}(1-|z|^2)^\alpha\left|\int_{0}^{1}\frac{t(1-t)^{1-\alpha}\left[(t-1)z+1+t\right]^{1-\alpha}}{(1-z)^\alpha\left[(t-1)z+1\right]^{3-2\alpha}}-\frac{t}{\left[(t-1)z+1\right](1-z)}dt\right|\\
    	&\geq\lim\limits_{r\rightarrow1^-}(1-r^2)^\alpha\left|\int_{0}^{1}\frac{t}{\left[(t-1)r+1\right](1-r)}-\frac{t(1-t)^{1-\alpha}\left[(t-1)r+1+t\right]^{1-\alpha}}{(1-r)^\alpha\left[(t-1)r+1\right]^{3-2\alpha}}dt\right|\\
    	&=\infty.
    \end{align*}
 	Then $\mathcal{H}$  is not bounded operator on  $\mathcal{B}^\alpha$ for $0<\alpha<1$.

 	For $\alpha>1$, we have that
 	\begin{align*}
 		\|\mathcal{H}\|_{\mathcal{B}^\alpha\rightarrow\mathcal{B}^\alpha}
 		&\geq\int_{0}^{1}\frac{1}{2(\alpha-1)}(1-t^2)^{1-\alpha}dt-\frac{1}{2(\alpha-1)}\\
 		&+\frac{1}{2(\alpha-1)}\lim\limits_{r\rightarrow1^-}(1-r^2)^\alpha\left|\int_{0}^{1}\frac{t}{\left[(t-1)r+1\right](1-r)}-\frac{t(1-t)^{1-\alpha}\left[(t-1)r+1+t\right]^{1-\alpha}}{(1-r)^\alpha\left[(t-1)r+1\right]^{3-2\alpha}}dt \right|	\\
 		&=\int_{0}^{1}\frac{1}{2(\alpha-1)}(1-t^2)^{1-\alpha}dt-\frac{1}{2(\alpha-1)}+\frac{1}{\alpha-1}\int_{0}^{1}\frac{(1-t)^{1-\alpha}}{t^{1-\alpha}}dt.
 	\end{align*}
 	 It is evident that $\mathcal{H}$ cannot  map $\mathcal{B}^\alpha$ to $\mathcal{B}^\alpha$ for $\alpha\geq2$.

 For $1<\alpha<2$, we obtain that
 	\begin{align*}
 		 \|\mathcal{H}\|_{\mathcal{B}^\alpha\rightarrow\mathcal{B}^\alpha}&\geq\int_{0}^{1}\frac{(1-t^2)^{1-\alpha}}{2(\alpha-1)}dt-\frac{1}{2(\alpha-1)}+\frac{1}{\alpha-1}B(2-\alpha,\alpha)\\
 		 &=\int_{0}^{1}\frac{(1-t^2)^{1-\alpha}}{2(\alpha-1)}dt-\frac{1}{2(\alpha-1)}+\frac{\Gamma(2-\alpha)\Gamma(\alpha-1)}{\Gamma(2)}\\
 		 &=\int_{0}^{1}\frac{(1-t^2)^{1-\alpha}}{2(\alpha-1)}dt -\frac{1}{2(\alpha-1)}+\frac{\pi}{\sin\pi(\alpha-1)}.
 	\end{align*}
 	
 	Last, we consider the condition for $\alpha=1$. Let $h_1(z)=\log\frac{1}{1-z}$, we know that $f_1\in\mathcal{B}$ and we obtain that
 	\begin{align*}
 		\|\mathcal{H}f_1(z)\|_{\mathcal{B}^1}&\geq\sup_{z\in\mathbb{D}}(1-|z|^2)\left|\int_{0}^{1}\frac{t}{\left[1+(t-1)z\right](1-z)}\log\frac{1+(t-1)z}{(1-t)(1-z)}dt\right|\\
 		&\geq\sup_{0\leq r<1}(1+r)\int_{0}^{1}\frac{t}{1+(t-1)r}\log\frac{1+(t-1)r}{(1-t)(1-r)}dt\\
 		&\geq\lim\limits_{r\rightarrow1^-}(1+r)\int_{0}^{1}\frac{t}{1+(t-1)r}\log\frac{1+(t-1)r}{(1-t)(1-r)}dt=\infty,
 	\end{align*}
 	This shows that $\mathcal{H}$ cannot  map $\mathcal{B}^1$ to $\mathcal{B}^1$. This finishes the proof of the theorem.
  \end{proof}
  Next, we consider the upper bound of the norm estimate for Hilbert matrix operator acting on $\mathcal{B}^\alpha$ for $1<\alpha<2$.
  \begin{theorem}
  	For $1<\alpha<2$, the norm of the Hilbert matrix operator acting on $\mathcal{B}^\alpha$ satisfies $$\|\mathcal{H}\|_{\mathcal{B}^\alpha\rightarrow\mathcal{B}^\alpha}\leq\frac{2^{\alpha}\pi}{\sin(\alpha-1)\pi}+\frac{1}{2-\alpha}.$$
  \end{theorem}
  \begin{proof}
  	Let $f\in\mathcal{B}^\alpha$ for $1<\alpha<2$. According to (\ref{eq2.1}), (\ref{eq5.1}) and Lemma 6.1,  we  obtain that
  	\begin{align*}
  		\|\mathcal{H}f\|_{\mathcal{B}^\alpha}&=|\mathcal{H}f(0)|+\sup_{z\in\mathbb{D}}|(\mathcal{H}f)'(z)|(1-|z|^2)^\alpha\\
  		&=\left|\int_{0}^{1}f(t)dt\right|+\sup_{z\in\mathbb{D}}(1-|z|^2)^\alpha\left|\int_{0}^{1}\frac{\phi_t(z)}{1-z}f(\phi_t(z))dt\right|\\
  		&\leq\int_{0}^{1}|f(t)|dt+\sup_{z\in\mathbb{D}}(1-|z|^2)^\alpha\int_{0}^{1}\frac{|\phi_t(z)|}{1-|z|}|f(\phi_t(z))|dt\\
  		&\leq\int_{0}^{1}\left(\frac{(1-t)^{1-\alpha}-1}{\alpha-1}\|f\|_{\alpha*}+|f(0)|\right)dt\\
  		&+\sup_{z\in\mathbb{D}}(1+|z|)^\alpha(1-|z|)^{\alpha-1}\int_{0}^{1}|\phi_(z)|\left[\frac{(1-\phi_t(z))^{1-\alpha}-1}{\alpha-1}\|f\|_{\alpha*}+|f(0)|\right]dt.
  	\end{align*}
  	
  	Noting that $|\phi_t(z)|=|\frac{t}{1+(t-1)z}|\leq\frac{t}{1+(t-1)|z|}$, we  get that
  	  	\begin{align*}
  		\|\mathcal{H}f\|_{\mathcal{B}^\alpha}&\leq\frac{1}{2-\alpha}\|f\|_{\alpha*}+|f(0)|\\
  		&+\sup_{z\in\mathbb{D}}(1+|z|)^\alpha(1-|z|)^{\alpha-1}\int_{0}^{1}\frac{t}{1+(t-1)|z|}\left(\frac{(1-t)^{1-\alpha}(1-|z|)^{1-\alpha}}{(\alpha-1)(1+(t-1)|z|)^{1-\alpha}}-\frac{1}{\alpha-1}\right)\|f\|_{\alpha*}dt.\\
  		&+\sup_{z\in\mathbb{D}}(1+|z|)^\alpha(1-|z|)^{\alpha-1}\int_{0}^{1}\frac{t}{1+(t-1)|z|}|f(0)|dt
  	\end{align*}
  	\begin{align*}
  		&\leq\frac{1}{2-\alpha}\|f\|_{\alpha*}+|f(0)|\\
  		&+\sup_{0\leq r<1}(1+r)^\alpha\int_{0}^{1}\frac{t}{1+(t-1)r}\left(\frac{(1-t)^{1-\alpha}}{(\alpha-1)(1+(t-1)r)^{1-\alpha}}-\frac{(1-r)^{\alpha-1}}{\alpha-1}\right)\|f\|_{\alpha*}dt.\\
  		&+\sup_{0\leq r<1}(1+r)^\alpha(1-r)^{\alpha-1}\int_{0}^{1}\frac{t}{1+(t-1)r}|f(0)|dt.
  	\end{align*}
It follows  that
 \begin{align*}
	&\sup_{0\leq r<1}(1+r)^\alpha\int_{0}^{1}\frac{t}{1+(t-1)r}\left(\frac{(1-t)^{1-\alpha}}{(\alpha-1)(1+(t-1)r)^{1-\alpha}}-\frac{(1-r)^{\alpha-1}}{\alpha-1}\right)dt\\
	&\leq\sup_{0\leq r<1}(1+r)^\alpha\int_{0}^{1}\frac{t}{1+(t-1)r}\cdot\frac{(1-t)^{1-\alpha}}{(\alpha-1)(1+(t-1)r)^{1-\alpha}}dt\\
	&-\inf_{0\leq r<1}(1+r)^\alpha\int_{0}^{1}\frac{t}{1+(t-1)r}\cdot\frac{(1-r)^{\alpha-1}}{\alpha-1}dt\\
&=\sup_{0\leq r<1}(1+r)^\alpha\int_{0}^{1}\frac{t(1-t)^{1-\alpha}}{(\alpha-1)(1+(t-1)r)^{2-\alpha}}dt-0\\
	&\leq 2^\alpha\int_{0}^{1}\frac{t(1-t)^{1-\alpha}}{(\alpha-1)(1-(1-t))^{2-\alpha}}dt=\frac{2^{\alpha}}{\alpha-1}
\int_{0}^{1}\frac{(1-t)^{1-\alpha}}{t^{1-\alpha}}dt=\frac{2^{\alpha}\pi}{\sin(\alpha-1)\pi}.
\end{align*}
Then
\begin{align*}
	\|\mathcal{H}f\|_{\mathcal{B}^\alpha}&\leq\left(\frac{2^{\alpha}\pi}{\sin(\alpha-1)\pi}+\frac{1}{2-\alpha}\right)\|f\|_{\alpha*}\\
	&+\left(\sup_{0\leq r<1}(1+r)^\alpha(1-r)^{\alpha-1}\int_{0}^{1}\frac{t}{1+(t-1)r}dt+1\right)|f(0)|.
\end{align*}
In another way, we can get that
\begin{align*}
	\sup_{0\leq r<1}(1+r)^\alpha(1-r)^{\alpha-1}\int_{0}^{1}\frac{t}{1+(t-1)r}dt\leq2^\alpha\int_{0}^{1}dt=2^\alpha<\frac{2^{\alpha}\pi}{\sin(\alpha-1)\pi}
\end{align*}
Therefore, we can get that$$\|\mathcal{H}f\|_{\mathcal{B}^\alpha\rightarrow\mathcal{B}^\alpha}\leq\frac{2^{\alpha}\pi}{\sin(\alpha-1)\pi}+\frac{1}{2-\alpha}.$$We complete the proof.
\end{proof}

\section{Norm of the Hilbert matrix $\|\mathcal{H}\|_{H^\infty\rightarrow \mathcal{B}}$}
In \cite{6}, it is established that the Hilbert matrix operator $\mathcal{H}$ is not bounded on $H^\infty$. However, it maps $H^\infty$ into the Bloch space $\mathcal{B}$. In this section, we compute the exact value of the norm of the Hilbert matrix from $H^\infty$ into $\mathcal{B}$.

\begin{theorem}
	The norm of the Hilbert matrix operator acting from $H^\infty$ into $\mathcal{B}$ satisfies $$\|\mathcal{H}\|_{H^\infty\rightarrow \mathcal{B}}=3.$$
\end{theorem}
\begin{proof}
	First, let's consider the upper bound of  $\|\mathcal{H}\|_{H^\infty\rightarrow \mathcal{B}}$. Suppose that $f\in H^\infty$, by (\ref{eq2.1}), we have that
	\begin{align*}
		\|\mathcal{H}f\|_\mathcal{B}&=|\int_{0}^{1}f(t)dt|+\sup_{z\in\mathbb{D}}(1-|z|^2)|(\mathcal{H}f)'(z)|\\                       	
		&= |\int_{0}^{1}f(t)dt|+\sup_{z\in\mathbb{D}}(1-|z|^2)|\int_{0}^{1}\frac{f(t)t}{(1-tz)^2}dt|\\
		&\leq\int_{0}^{1}|f(t)|dt+\sup_{z\in\mathbb{D}}(1-|z|^2)\int_{0}^{1}\frac{|f(t)t|}{(1-t|z|)^2}dt\\
		&=\int_{0}^{1}|f(t)|dt+\sup_{0\leq r<1}(1-r^2)\int_{0}^{1}\frac{|f(t)t|}{(1-tr)^2}dt\\
		&\leq \left(1+\sup_{0\leq r<1}(1-r^2)\int_{0}^{1}\frac{1}{(1-tr)^2}dt\right)\|f\|_{H^\infty}\\
		&=\left(1+\sup_{0\leq r<1}\frac{1-r^2}{r}\int_{0}^{r}\frac{1}{(1-t)^2}dt\right)\|f\|_{H^\infty}\\
		&=3\|f\|_{H^\infty}.
	\end{align*}
Therefore, we have obtained the upper bound.  Now  we consider the lower bound of $\|\mathcal{H}\|_{H^\infty\rightarrow \mathcal{B}}$. According to \cite{1}, let $f=1\in H^\infty$, then $$\|f\|_{H^\infty}=1,\mathcal{H}(1)(z)=\frac{1}{z}\log\frac{1}{1-z},\mathcal{H}(1)(0)=1.$$
We have that
\begin{align*}
	\|\mathcal{H}\|_{H^\infty\rightarrow \mathcal{B}}&\geq \mathcal{H}(1)(0)+\sup_{z\in\mathbb{D}}(1-|z|^2)|\mathcal{H}(1)(z)|\\
		&=1+\sup_{z\in\mathbb{D}}(1-|z|^2)\big|\frac{1}{z(1-z)}-\frac{1}{z^2}\log\frac{1}{1-z} \big|\\
	&\geq1+\sup_{0\leq r<1}(1-r^2)\big|\frac{1}{r(1-r)}-\frac{1}{r^2}\log\frac{1}{1-r}\big|\\
	&=1+\sup_{0\leq r<1}\big|\frac{1+r}{r}-\frac{1-r^2}{r^2}\log\frac{1}{1-r}\big|\\
	&\geq 1+\lim\limits_{r\rightarrow1}(\frac{r+1}{r}-\frac{1-r^2}{r^2}\log\frac{1}{1-r})\\
	&=3.
\end{align*}

We obtain the result by estimating the upper and lower bounds.
\end{proof}
\section{Norm estimates of the Hilbert matrix $\|\mathcal{H}\|_{H^\infty_\alpha\rightarrow \mathcal{B}^{\alpha+1}}$}
 In this section,  we delve into the norm estimates for the Hilbert matrix operator $\mathcal{H}$ as it acts from the Korenblum space $H^\infty_\alpha$ into the $(\alpha+1)$-Bloch space $\mathcal{B}^{\alpha+1}$ for $0 <\alpha<1$.


\begin{theorem}
	Let $0<\alpha\leq\frac{2}{3}$. Then
	$$\|\mathcal{H}\|_{H_\alpha^\infty\rightarrow\mathcal{B}^{\alpha+1}}=\int_{0}^{1}\frac{1}{(1-t^2)^\alpha}dt+\frac{2\alpha\pi}{\sin\alpha\pi}$$
	
For $\frac{2}{3}<\alpha<1$, we have the following lower and upper bound: $$\int_{0}^{1}\frac{1}{(1-t^2)^\alpha}dt+\frac{2\alpha\pi}{\sin\alpha\pi}\leq\|\mathcal{H}\|_{H_\alpha^\infty\rightarrow\mathcal{B}^{\alpha+1}}\leq\int_{0}^{1}\frac{1}{(1-t^2)^\alpha}dt+2\int_{0}^{1}\frac{(1-t)^{1-\alpha}}{t^\alpha}dt$$

For  $\alpha\geq 1$, then $\mathcal{H}: H^\infty_\alpha\rightarrow \mathcal{B}^{\alpha+1}$ is not a bounded operator.
	
\end{theorem}
\begin{proof}
	First, let's consider the lower bound of $\|\mathcal{H}\|_{H_\alpha^\infty\rightarrow\mathcal{B}^{\alpha+1}}$. Let $0<\alpha<1$ and $z\in\mathbb{D}$. Define$$f_\alpha(z)=\frac{1}{(1-z^2)^\alpha}.$$
	On one hand, we have the estimate $$\|f_\alpha\|_{H_\alpha^\infty}=\sup_{z\in\mathbb{D}}\frac{(1-|z|^2)^\alpha}{|1-z^2|^\alpha}\leq\sup_{z\in\mathbb{D}}\frac{(1-|z|^2)^\alpha}{(1-|z|^2)^\alpha}=1.$$
	On the other hand, for $r\in(0,1)$ it holds that $\lim_{r\rightarrow1^-}|f_\alpha(r)|(1-r^2)^\alpha=1$, and we obtain $\|f_\alpha\|_{H_\alpha^\infty}=1$.
	
	According to (\ref{eq2.1}) and (\ref{eq5.2}), we obtain that
	\begin{align*}
		\|\mathcal{H}\|_{H_\alpha^\infty\rightarrow\mathcal{B}^{\alpha+1}}&\geq\frac{\|\mathcal{H}(f_\alpha)\|_{\mathcal{B}^{\alpha+1}}}{\|f_\alpha\|_{H_\alpha^\infty}}\\
		&=|\mathcal{H}f(0)|+\sup_{z\in\mathbb{D}}(1-|z|^2)^{\alpha+1}|(\mathcal{H}f)'(z)|\\
		&=\int_{0}^{1}\frac{1}{(1-t^2)^\alpha}dt+\sup_{z\in\mathbb{D}}(1-|z|^2)^{\alpha+1}\left|\int_{0}^{1}\frac{t}{\left[(t-1)z+1\right](1-z)(1-\phi_t(z)^2)^\alpha}dt\right|\\
		&=\int_{0}^{1}\frac{1}{(1-t^2)^\alpha}dt+\sup_{z\in\mathbb{D}}(1-|z|^2)^{\alpha+1}\left|\int_{0}^{1}\frac{t\left[(t-1)z+1\right]^{2\alpha-1}}{(1-z)^{\alpha+1}(1-t)^\alpha\left[(t-1)z+1+t\right]^\alpha}dt\right|\\
		&\geq\int_{0}^{1}\frac{1}{(1-t^2)^\alpha}dt+\sup_{0\leq r<1}(1-r^2)^{\alpha+1}\left|\int_{0}^{1}\frac{t\left[(t-1)r+1\right]^{2\alpha-1}}{(1-r)^{\alpha+1}(1-t)^\alpha\left[(t-1)r+1+t\right]^\alpha}dt\right|\\
		&\geq\int_{0}^{1}\frac{1}{(1-t^2)^\alpha}dt+\lim\limits_{r\rightarrow1^-}(1-r^2)^{\alpha+1}\left|\int_{0}^{1}\frac{t\left[(t-1)r+1\right]^{2\alpha-1}}{(1-r)^{\alpha+1}(1-t)^\alpha\left[(t-1)r+1+t\right]^\alpha}dt\right|\\
		&=\int_{0}^{1}\frac{1}{(1-t^2)^\alpha}dt+\int_{0}^{1}\frac{2t^\alpha}{(1-t)^\alpha}dt\\
		&=\int_{0}^{1}\frac{1}{(1-t^2)^\alpha}dt+2B(1+\alpha,1-\alpha)\\
		&=\int_{0}^{1}\frac{1}{(1-t^2)^\alpha}dt+\frac{2\Gamma(1+\alpha)\Gamma(1-\alpha)}{\Gamma(2)}\\
		&=\int_{0}^{1}\frac{1}{(1-t^2)^\alpha}dt+\frac{2\alpha\Gamma(\alpha)\Gamma(1-\alpha)}{\Gamma(2)}\\
		&=\int_{0}^{1}\frac{1}{(1-t^2)^\alpha}dt+\frac{2\alpha\pi}{\sin\pi\alpha}.
	\end{align*}
	Therefore, we have obtained the lower bound.

 Now we  consider the upper bound of $\|\mathcal{H}\|_{H_{\alpha}^\infty\rightarrow \mathcal{B}^{\alpha+1}}$. Suppose that $f\in H^\infty_\alpha$, it follows that
	\begin{align*}
		\|\mathcal{H}f\|_{\mathcal{B}^{\alpha+1}}&=|\mathcal{H}f(0)|++\sup_{z\in\mathbb{D}}(1-|z|^2)^{\alpha+1}|(\mathcal{H}f)'(z)|\\
		&=|\int_{0}^{1}f(t)dt|+\sup_{z\in\mathbb{D}}(1-|z|^2)^{\alpha+1}\left|\int_{0}^{1}\frac{t}{\left[(t-1)z+1\right](1-z)}f(\phi_t(z))dt\right|\\
		&=|\int_{0}^{1}\frac{(1-t^2)^\alpha}{(1-t^2)^\alpha}f(t)dt|+\sup_{z\in\mathbb{D}}(1-|z|^2)^{\alpha+1}\left|\int_{0}^{1}\frac{\phi_t(z)}{(1-z)}\frac{(1-\phi^2_t(z))^\alpha}{(1-\phi^2_t(z))^\alpha}f(\phi_t(z))dt\right|\\
		&\leq\int_{0}^{1}\frac{1}{(1-t^2)^\alpha}dt\|f\|_{H_\alpha^\infty}+\sup_{z\in\mathbb{D}}(1-|z|^2)^{\alpha+1}\left|\int_{0}^{1}
\frac{\phi_t(z)}{(1-z)}\frac{1}{(1-\phi^2_t(z))^\alpha}dt\right|\|f\|_{H_\alpha^\infty}\\
		&\leq\int_{0}^{1}\frac{1}{(1-t^2)^\alpha}dt\|f\|_{H_\alpha^\infty}+\sup_{z\in\mathbb{D}}(1-|z|^2)^{\alpha+1}\int_{0}^{1}
\frac{|\phi_t(z)|}{|1-z|}\frac{1}{|1-\phi^2_t(z)|^\alpha}dt\|f\|_{H_\alpha^\infty}.
	\end{align*}
	Noting that  $|\phi_t(z)|\leq\phi_t(|z|)$, we can obtain that
	\begin{align*}
		\|\mathcal{H}\|_{H_\alpha^\infty\rightarrow\mathcal{B}^{\alpha+1}}
		&\leq\int_{0}^{1}\frac{1}{(1-t^2)^\alpha}dt+\sup_{z\in\mathbb{D}}(1-|z|^2)^{\alpha+1}\int_{0}^{1}\frac{\phi_t(|z|)}{|1-z|}\frac{1}{(1-\phi^2_t(|z|))^\alpha}dt\\
		&=\int_{0}^{1}\frac{1}{(1-t^2)^\alpha}dt+\sup_{0\leq r<1}(1+r)^{\alpha+1}\int_{0}^{1}\frac{t\left[(t-1)r+1\right]^{2\alpha-1}}{(1-t)^{\alpha}\left[(t-1)r+1+t\right]^\alpha}dt.
	\end{align*}
 Changing the variable of integration, we obtain that
	$$\|\mathcal{H}\|_{H_\alpha^\infty\rightarrow\mathcal{B}^{\alpha+1}}
	\leq\int_{0}^{1}\frac{1}{(1-t^2)^\alpha}dt+\sup_{0\leq r<1}(1+r)^{\alpha+1}\int_{0}^{1}\frac{(1-t)(1-rt)^{2\alpha-1}}{t^{\alpha}\left[2-(1+r)t\right]^\alpha}dt.$$
	
	For $0<\alpha\leq\frac{2}{3}$, let $g(r,t)=\frac{(1-rt)^{2\alpha-1}}{\left(2-(1+r)t\right)^\alpha}$, and we have that
	\begin{align*}
		\frac{\partial g(r,t)}{\partial r}&=\frac{t(1-rt)^{2\alpha-2}}{\left(2-(1+r)t\right)^{\alpha+1}}\left[(1-\alpha)(1-rt)+(1-2\alpha)(1-t)\right]\\
		&\geq\frac{t(1-rt)^{2\alpha-2}}{\left(2-(1+r)t\right)^{\alpha+1}}\left[(1-\alpha)(1-t)+(1-2\alpha)(1-t)\right]\\
		&=\frac{t(1-rt)^{2\alpha-2}}{\left(2-(1+r)t\right)^{\alpha+1}}(2-3\alpha)(1-t)\\
		&\geq 0.
	\end{align*}
	Therefore, $g$ is monotonically increasing on $r\in(0,1)$. It follows that
	\begin{align*}
		\|\mathcal{H}\|_{H_\alpha^\infty\rightarrow\mathcal{B}^{\alpha+1}}
		&\leq\int_{0}^{1}\frac{1}{(1-t^2)^\alpha}dt+\sup_{0\leq r<1}(1+r)^{\alpha+1}\int_{0}^{1}\frac{(1-t)(1-rt)^{2\alpha-1}}{t^{\alpha}\left[2-(1+r)t\right]^\alpha}dt\\
		&=\int_{0}^{1}\frac{1}{(1-t^2)^\alpha}dt+\lim\limits_{r\rightarrow1^-}(1+r)^{\alpha+1}\int_{0}^{1}\frac{(1-t)(1-rt)^{2\alpha-1}}{t^{\alpha}\left[2-(1+r)t\right]^\alpha}dt\\
		&=\int_{0}^{1}\frac{1}{(1-t^2)^\alpha}dt+2\int_{0}^{1}\frac{(1-t)^\alpha}{t^\alpha}dt\\
		&=\int_{0}^{1}\frac{1}{(1-t^2)^\alpha}dt+2B(1+\alpha,1-\alpha)\\
		&=\int_{0}^{1}\frac{1}{(1-t^2)^\alpha}dt+\frac{2\alpha\pi}{\sin\pi\alpha}.
	\end{align*}
Thus, for $0 <\alpha\leq\frac{2}{3}$,
	$$\|\mathcal{H}\|_{H_\alpha^\infty\rightarrow\mathcal{B}^{\alpha+1}}=\int_{0}^{1}\frac{1}{(1-t^2)^\alpha}dt+\frac{2\alpha\pi}{\sin\alpha\pi}.$$
	For $\frac{2}{3}<\alpha<1$, we have that
	\begin{align*}
		\|\mathcal{H}\|_{H_\alpha^\infty\rightarrow\mathcal{B}^{\alpha+1}}
		&\leq\int_{0}^{1}\frac{1}{(1-t^2)^\alpha}dt+\sup_{0\leq r<1}(1+r)^{\alpha+1}\int_{0}^{1}\frac{(1-t)(1-rt)^{2\alpha-1}}{t^{\alpha}\left[2-(1+r)t\right]^\alpha}dt\\
		&\leq\int_{0}^{1}\frac{1}{(1-t^2)^\alpha}dt+2\int_{0}^{1}\frac{(1-t)^{1-\alpha}}{t^\alpha}dt.
	\end{align*}
From the above proof, we know that for $\alpha\geq1$, we define $f_\alpha(z)=\frac{1}{(1-z^2)^\alpha}$, and $f_\alpha\in H_\alpha^\infty$.  Then $\mathcal{H}f(0)$ is not defined. Therefore, $\mathcal{H}$ cannot map $H_\alpha^\infty$ to $\mathcal{B}^{\alpha+1}$.
	This finishes the proof of the theorem.
\end{proof}

\end{document}